\documentclass[english,12pt]{article}
 \usepackage[latin1]{inputenc}
  \usepackage{amsmath,amsthm,amssymb,babel}

\usepackage{color}

\newtheorem{theorem}{Theorem}
 
 \newtheorem{step}{Step}
  
   \newtheorem{example}{Example}
    
\newtheorem{proposition}{Proposition}
 \newtheorem{problem}{Problem}
  \newtheorem{assumption}{Assumption}

    \newcommand{\bnu}{\mbox{\boldmath{$\nu$}}}

  \newcommand{\bx}{\mbox{\boldmath{$x$}}}

  \date{ }

\begin{document}

\title{\bf Solvability and  optimization for a class of mixed variational problems}

  {\author{ Andaluzia Matei$^{1}$ \footnote{This is a preprint. To cite the final version: Andaluzia Matei, Mircea Sofonea (2019) Solvability and optimization for a class of mixed variational problems, Optimization, DOI: 10.1080/02331934.2019.1676242
}\,  and \, {Mircea Sofonea$^{2}$}\\[5mm]
  		{\it \small $^1$ Department of Mathematics, University of Craiova}\\ [-5mm]
  		{\it\small	A.I. Cuza 13, 200585, Craiova, Romania}\\[-5mm]
  		{\it\small	E-mail:andaluziamatei@inf.ucv.ro}\\[0mm]
  		{\it \small $^2$ Laboratoire de Math\'ematiques et Physique}\\ [-5mm]
  		{\it \small	University of Perpignan Via Domitia}\\[-5mm]
  		{\it\small 52 Avenue Paul Alduy, 66860 Perpignan, France}\\ [-5mm]
  		{\it\small	E-mail: sofonea@univ-perp.fr}	
}}

 \maketitle

\begin{abstract}

We consider an abstract mixed variational problem governed by a nonlinear operator $A$ and a bifunctional $J$, in a real reflexive Banach space $X$. The operator $A$ is assumed to be continuous, Lipschitz continuous on each bounded subset of $X,$ and generalized monotone. First, we pay attention to the unique solvability of the problem. Next,	we prove a continuous dependence result of the solution with respect to the data. Based on this result we prove the existence of at least one solution for an associated optimization problem. Finally, we apply our abstract results to the well-posedness and the optimization of an antiplane frictional contact model for nonlinearly elastic materials of Hencky-type.
\end{abstract}


\vskip5mm


\section{Introduction}

Mixed variational formulations arise in the analysis of various nonlinear boundary value problems  which appear in
Solid, Fluid and Contact Mechanics, and in various Engineering Applications, as well. Their solvability is based on arguments of saddle point, monotonicity, convexity and fixed-point techniques, among others. Existence and uniqueness results can be found in
\cite{Cea, GLT,HHN96,HHNL,Reddy92}, for instance.  Owing a specific structure and involving Lagrange multipliers, the numerical treatment of
mixed variational problems  is efficient and accurate. Reference in the field
include \cite{AHLMR,BMS,HR,BRB,HMW}, among others.

Recently, in \cite{SMX} we considered a mixed variational problem in a real Hilbert space $X$, governed by a strongly monotone  Lipschitz continuous operator $A$ and a completely continuous operator $\pi.$ The main result in \cite{SMX} was  the continuous dependence of the solution with respect to the data. A more general mixed variational problem
was considered in \cite{CAMWA}, in the framework of  a real reflexive Banach space. There, the existence of at least one solution of the problem was proved, under the assumptions that the operator $A$ is generalized monotone and hemicontinuous.

The present paper represents a continuation of our previous papers \cite{CAMWA,SMX}. Its main novelty arises in the fact that here we extend the continuous dependence  result in \cite{SMX} to the generalized mixed variational problem in \cite{CAMWA}, assuming that the operator $A$ is continuous, Lipschitz continuous on each bounded subset of $X,$ and generalized monotone. Note that such an operator could fail to be globally Lipschitz continuous. Following a technique developed in \cite{SMX}, we apply our new results to study  the solvability of an associated optimization problem. Finally, we use the abstract results in the study of the well-posedness and  optimization of a  nonlinear boundary value problem which describes an antiplane frictional contact problem with elastic materials of Hencky-type, see, e.g., \cite{MMN18} and the references therein.

\medskip
The abstract problem under consideration can be stated as follows.

\begin{problem}\label{ap} Find $u\in X$ and $\lambda\in \Lambda\subset Y$ such that
\begin{eqnarray}
(Au,v-u)_{X',X}+b(v-u,\lambda)+J(u,v-u)&\geq& (f,\,\pi(v-u))_{Z',Z}\quad
    \mbox{\rm for all }\, v\in X,\label{pp:1}\\[2mm]
b(u,\mu-\lambda)&\leq& 0\qquad\qquad\qquad\,\,\,\,\quad\mbox{\rm for all }\, \mu\in \Lambda.\label{pp:2}
 \end{eqnarray}
\end{problem}
\noindent

\medskip
Here $(X,\|\cdot\|_X),$  $(Y,\|\cdot\|_Y)$  and $(Z,\|\cdot\|_Z)$ are real reflexive Banach spaces,
$(\cdot,\cdot)_{X',X}$ denotes
the duality pairing between $X$ and its dual $X'$ and  $(\cdot,\cdot)_{Z',Z}$ denotes
the duality pairing between $Z$ and its dual $Z'$.
Moreover, $A:X\to X'$, $b:X\times Y\to \mathbb R$, $J:X\times X\to\mathbb{R}$, $f\in Z'$, $\pi:X\to Z$ and $\Lambda\subset Y$ are given. Below in this paper we  use $``\to"$ and $``\rightharpoonup"$ for the strong and weak convergence in various normed spaces that will be specified and, unless stated otherwise, all the limits, upper limits and lower limits  are considered when $n\to+\infty$.

\medskip
In the study of Problem \ref{ap} we consider the following assumptions.


\begin{assumption}\label{a4} There
exists a functional
$h:X\rightarrow \mathbb R$ such that:

$(i_1)$\ $h(tw)=t^r h(w)$ for all
$t>0, \,w\in X$, with a given $r>1;$

$(i_2)$\ $(Av-Au,v-u)_{X',X}\geq h(v-u)$ for all $u,v\in X;$

$(i_3)$\ if $\{u_n\}\subset X$ is a sequence such that $u_n\rightharpoonup u$ in $X$, then $h(u)\leq \limsup h(u_n).$
\end{assumption}

\begin{assumption}\label{a4bis} The operator $A:X\to X'$ is continuous.
\end{assumption}


\begin{assumption}\label{coercive}
	\ $\displaystyle\frac{(Au,u)_{X',X}}{\|u\|_X}\to \infty\quad\mbox{ \rm as }\quad\|u\|_X\to \infty.$
\end{assumption}

 \begin{assumption} \label{a5}The form $b:X\times Y\to \mathbb R$ is bilinear. In addition:

 $(i_1)$  for each sequence $\{u_n\}\subset X$ such that $u_n\rightharpoonup u$ in $X$  we have $b(u_n,\mu)\to b(u,\mu)$,  for all $\mu\in Y;$

$(i_2)$ for each sequence $\{\lambda_n\}\subset Y$ such that $\lambda_n\rightharpoonup \lambda$ in $Y$ we have $b(v,\lambda_n)\to b(v,\lambda)$, for all $v\in X$;

$(i_3)$  $b$ satisfies the so-called ``inf-sup" condition, i.e., there exists $\alpha> 0$ such that \begin{equation}\displaystyle\inf_{\mu\in Y, \mu \neq 0_{Y}}\,\sup_{v\in X, v\neq 0_{X}}\,\displaystyle\frac{b(v,\mu)}{\|v\|_{X}\,\|\mu\|_{Y}}\geq \alpha.\label{is}
\end{equation}
\end{assumption}
\begin{assumption}\label{J}  The function $J:X\times X\to\mathbb{R}$ is such that:

$(i_1)$\	for every $u\in X,$ the application $X \ni v\to J(u,v)\in \mathbb R$ is
positively homogeneous and subadditive on $X$;

$(i_2)$\ there exists $c>0$  such that
\begin{equation}\label{iyu}
|J(u,v)|\leq c\|v\|_X\qquad\mbox{\rm  for all }\ u,\, v\in X;
\end{equation}

$(i_3)$\
the application $X\times X \ni (u,v)\to J(u,v)\in \mathbb R$ is weakly upper semicontinuous, i.e.,
if $\{u_n\}\subset X$,  $\{v_n\}\subset X$ are such that  $u_n\rightharpoonup u$, $v_n\rightharpoonup v$ in $X$,  then \[\limsup J(u_{n},v_{n})\leq J(u,v).\]
\end{assumption}
\begin{assumption}\label{ef}\  $f\in Z'$.
\end{assumption}
\begin{assumption}\label{a3bbiss} $\Lambda$ is a closed convex bounded subset of $Y$ such that $0_Y\in \Lambda$.
\end{assumption}

\begin{assumption}\label{pi} The operator $\pi$ is a linear and continuous operator.
\end{assumption}

The following existence result is a straightforward consequence of an existence result obtained in the recent paper \cite{CAMWA}.
\begin{theorem}\label{maini}
	Under  Assumptions $\ref{a4}$--$\ref{pi}$  there exists at least one solution to Problem $\ref{ap}$.
\end{theorem}

Note that Theorem \ref{maini} guarantees the solvability of Problem $\ref{ap}$. Nevertheless, it leaves open a number of questions like
the uniqueness of the solution and its continuous dependence on the data $A$, $b$, $J$, $f$, $\Lambda$, which represent crucial tools in the study of associated optimization and optimal control problems.
The aim of this paper is to fill this gap. Thus, in Section \ref{s2} we provide sufficient assumptions on the data which guarantee the uniqueness of the solution to  Problem $\ref{ap}$. Moreover, we study the boundedness of the solution. We use these results in
Section \ref{cr} where we prove a continuous dependence result of the solution with respect to the data. The proof follows the technique in \cite{SMX}, based on arguments of monotonicity, compactness and Mosco convergence.  Then, we use this continuous dependence result
to provide the existence of minimizers for an associated optimization problem. Our abstract results in this paper can be applied in the study of a large number of boundary value problems. To provide an example, we use them in Section \ref{apli}
in the study of a nonlinear problem governed by the $r$-Laplace operator.

We end this section recalling that comprehensive results on optimization and optimal control theory can be found in \cite{Ba,Cea,Li68,NST,Ti,Tiba}. For various results concerning the optimal control of variational and hemivariational inequalities we refer the reader  to, e.g., \cite{F,LZ,Mig,MP, S2}. The current paper completes part of the results in the aforementioned references, since here we deal with the minimization of cost functionals associated to mixed variational problems.

\section{Uniqueness and bounds}
\label{s2}

In this section we provide the uniqueness of the solution $(u,\lambda)$ of Problem \ref{ap} together with bounds for $u$ and $\lambda$ too. To this end, we consider the following additional assumptions.
\begin{assumption}\label{a7}
There exist $M>0$ and $q\geq 2$ such that
\begin{equation}
 h(v)\geq M\,\|v\|^q_X\,\,\,\,\,\,\,\mbox{\rm for all }\,\,v \in X.\label{hh}
\end{equation}
\end{assumption}

\begin{assumption}\label{a7bis} There exists $m\geq 0$ such that
\begin{equation}\label{iJ}
J(w_1,w_2-w_1)+J(w_2,w_1-w_2)\leq m\|w_1-w_2\|_X^q\quad \mbox{ \rm for all \ }w_1,\,w_2\in X,
\end{equation}
where $q$ is the constant which appears in $(\ref{hh})$.
\end{assumption}

\begin{assumption}\label{sa} $M > m$.
\end{assumption}

Our first result in this section is the following.

\begin{theorem}\label{ufa}
Under  Assumptions 
$\ref{a4}$--$\ref{sa}$
the solution of
Problem $\ref{ap}$ is unique  in its first argument.
\end{theorem}
\begin{proof}
Let  $(u_1,\lambda_1)$ and $(u_2,\lambda_2)$ be two solutions  of Problem $\ref{ap}$. We write
(\ref{pp:1}) with $u=u_1$,
$\lambda=\lambda_1$ and $v=u_2$,  then with $u=u_2$,
$\lambda=\lambda_2$ and $v=u_1$. We add the resulting inequalities to obtain that
\begin{eqnarray*}
&&(Au_1-A u_2,u_1-u_2)_{X',X}\le b(u_2,\lambda_1-\lambda_2)+b(u_1,\lambda_2-\lambda_1)\\ [2mm]
&&\qquad+J(u_1,u_2-u_1)+J(u_2,u_1-u_2).
\end{eqnarray*}
Next, since (\ref{pp:2}) implies that $b(u_2,\lambda_1-\lambda_2)\le 0$ and $b(u_1,\lambda_2-\lambda_1)\le 0$, we deduce that
\begin{eqnarray*}
	&&(Au_1-A u_2,u_1-u_2)_{X',X}\le J(u_1,u_2-u_1)+J(u_2,u_1-u_2).
\end{eqnarray*}
This last inequality together with Assumptions \ref{a4} $(i_2)$, \ref{a7} and \ref{a7bis} yields
\[(M-m)\|u_1-u_2\|_X^q\leq 0.\]
Therefore, Assumption \ref{sa} implies that $u_1=u_2$.
\end{proof}

In order to prove the uniqueness in the second argument, we need the following additional assumption.


\begin{assumption}\label{unic}
$J(u,v)+J(u,-v)\leq0\quad\mbox{\rm  for all }\ u,\,v\in X.$
\end{assumption}

Our second result in this section is the following.

\begin{theorem}\label{t1}  Under Assumptions $\ref{a4}$--$\ref{unic}$, 
Problem $\ref{ap}$ has a unique solution.
\end{theorem}
\begin{proof}
Let  $(u_1,\lambda_1)$ and $(u_2,\lambda_2)$ be two solutions of Problem $\ref{ap}$ and let $w\in X$, $w\ne 0_X$. We write (\ref{pp:1}) with $u=u_1$,
$\lambda=\lambda_1$ and $v=-w+u_1$,  then with $u=u_2$,
$\lambda=\lambda_2$ and $v=w+u_2.$ By adding the resulting inequalities we obtain
\[b(w,\lambda_1-\lambda_2)\leq \|Au_1-A u_2\|_{X'}\|w\|_X+J(u_1,-w)+J(u_2,w). \]
Since Theorem \ref{ufa} guarantees that $u_1=u_2$, using Assumption \ref{unic} we are led to
  \begin{equation}\label{100}
  \frac{b(w,\lambda_1-\lambda_2)}{\|w\|_X}\leq 0.
  \end{equation}
 Moreover, (\ref{is}) implies that
 \begin{equation}\label{101}
 \alpha\|\lambda_1-\lambda_2\|_Y\leq \operatorname{sup}_{w\in X,\,w\neq 0_X} \frac{b(w,\lambda_1-\lambda_2)}{\|w\|_X}.
 \end{equation}
We now combine inequalities (\ref{100}) and (\ref{101}) to deduce that
$\|\lambda_1-\lambda_2\|_Y\leq 0$,
which concludes the proof.
\end{proof}

We proceed with some boundedness results for the solution $(u,\lambda)$ of Problem $\ref{ap}$.

\begin{proposition}\label{pr1} Under Assumptions $\ref{a4}$--$\ref{sa}$ we have
\[\|u\|_X\leq M^{1/(1-q)}(c_0\|f\|_{Z'}+\|A0_X\|_{X'}+c)^{1/(q-1)},\]
$c$ being the positive constant in $(\ref{iyu})$.
\end{proposition}
\begin{proof}
We test with $v=0_X$ in (\ref{pp:1}) to obtain that
\begin{equation}\label{z1}
(Au,u)_{X',X}\leq b(-u,\lambda)+J(u,-u)+(f,\pi u)_{Z',Z}.
\end{equation}
Next, setting $\mu=0_Y$ in (\ref{pp:2}) we get
\begin{equation}\label{z2}
b(-u,\lambda)\leq 0.
\end{equation}
On the other hand, according to (\ref{iyu}) we have
\begin{equation}\label{z3}J(u,-u)\leq c\|u\|_X
\end{equation}
and, since the operator $\pi$ is a linear and continuous operator, there exists $c_0>0$ such that
\begin{equation}\label{c0}
 \|\pi v\|_{Z} \leq c_0\,\|v\|_X\quad \mbox{ for all }\,v\in X.
 \end{equation}
We now combine inequalities (\ref{z1})--(\ref{c0}) and use
Assumption \ref{a4} ($i_2$) and (\ref{hh}) to see that
\begin{equation}\label{nrmu}
M\|u\|_X^{q-1}\leq c_0\|f\|_{Z'}+\|A0_X\|_{X'}+c
\end{equation}
which concludes the proof.
\end{proof}

Next, we introduce the  bounded set
\begin{equation}\label{K1}
K_1=\{v\in X \mid \|v\|_X\leq M_1\},
\end{equation}
where
\begin{equation}\label{M1}
M_1=M^{1/(1-q)}(c_0\|f\|_{Z'}+\|A0_X\|_{X'}+c)^{1/(q-1)}.
\end{equation}
Note that, under the assumptions of Proposition \ref{pr1} we have $u\in K_1$.
Consider now  the following assumption.
\begin{assumption}\label{a10}
For each nonempty bounded subset $S\subset X$,  there exists $L_{S}>0$ such that
\begin{equation}\label{LA}
\|Au-Av\|_{X'}\leq L_S\,\|u-v\|_{X}\quad\mbox{\rm for all }\ \ u, v\in S.
\end{equation}
\end{assumption}

Then, the following result holds.

\begin{proposition} Under Assumptions $\ref{a4}$--$\ref{a10}$,
\begin{equation}\label{majla}
\|\lambda\|_Y\leq \displaystyle\frac{1}{\alpha}\,(c_0\|f\|_{Z'}+L_{K_1}\,M_1+\|A0_X\|_{X'}+c),
\end{equation}
$c_0$ being the positive constant in $(\ref{c0})$, $L_{K_1}$  the positive constant in $(\ref{LA})$ corresponding to the set $S=K_1$,
 and $c$ the positive constant in $(\ref{iyu})$.
\end{proposition}
\begin{proof}
Let $w\in X$ be arbitrarily fixed. Setting $v=u-w$ in (\ref{pp:1})
 we have
 \begin{eqnarray*}
 b(w,\lambda)&\leq& (Au,-w)_{X',X}+J(u,-w)+(f,\pi w)_{Z',Z}\\
 &\leq& \|Au\|_{X'}\|w\|_X+c\|w\|_X+c_0\|f\|_{Z'}\|w\|_X.
 \end{eqnarray*}
 Now, we use the inf-sup property of the form $b$, see Assumption $\ref{a5}$ $(i_3)$, to write
\begin{equation}\label{nr}
 \alpha\|\lambda\|_Y\leq c_0\|f\|_{Z'}+\|Au\|_{X'}+c.
 \end{equation}
On the other hand, since $u\in K_1,$ it follows that
\begin{equation}\label{zr}
\|Au\|_{X'}\leq L_{K_1}M_1+\|A0_X\|_{X'}.
\end{equation}
Combining now (\ref{nr}) and (\ref{zr}) we obtain
\[ \alpha\|\lambda\|_Y\leq L_{K_1} \,M_1+\|A0_X\|_{X'}+c+c_0\|f\|_{Z'}\]
which implies (\ref{majla}) and concludes the proof.
\end{proof}

\section{Convergence and optimization}\label{cr}

In the first part of this section we  study the dependence of the solution of Problem \ref{ap} with respect to the data and prove a convergence result. This  convergence result will be applied in the second part of this section in order to study an associated optimization problem.
We suppose in what follows that Assumptions \ref{a4}--\ref{a10} hold and we denote by $(u,\lambda)$ the unique solution of Problem \ref{ap} guaranteed by Theorem \ref{t1}. Moreover,
for each $n\in\mathbb{N}$ we consider  the following problem.

\begin{problem}\label{ppn} Find $(u_n,\lambda_n)\in X\times\Lambda_n$
	such that
	\begin{eqnarray}
	\label{1n}(A_n u_n,v-u_n)_{X',X}+b_n(v-u_n,\lambda_n)+J_n(u_n,v-u_n)&\geq& (f_n,\pi (v-u_n))_{Z',Z}
\\[2mm]
	b_n(u_n,\mu-\lambda_n)&\leq& 0.
\label{2n}
	\end{eqnarray}
for all $v\in X,$ $\mu\in \Lambda_n$.
\end{problem}
Herein, for each $n\in\mathbb{N}$, the operator $A_n$, the form $b_n$, the function $J_n,$ the element $f_n$ and the set $\Lambda_n$ represent a perturbation of $A$, $b$, $J$, $f$, $\Lambda$ and are supposed to
satisfy  Assumptions \ref{a4}--\ref{a3bbiss}, \ref{a7}--\ref{a10} with function $h_n$ and constants $\alpha_n,$ $c_n,$ $M_n,$ $q_n,$ $m_n$, $L_S^n$.  To avoid any confusion, when used with $n$, we refer to these assumptions as Assumptions $\ref{a4}_n$--$\ref{a3bbiss}_n$, $\ref{a7}_n$--$\ref{a10}_n$. Recall that Theorem \ref{t1} guarantees the uniqueness of the solution of Problem \ref{ppn}, denoted by $(u_n,\lambda_n)$.
To proceed, we consider the following additional assumptions.
\begin{assumption}\label{cvA}
There exists $\widetilde{\delta}$ and, for each $n\in\mathbb{N}$, there exist $F_n\ge 0$ and  $\delta_n\geq 0$  such that:

$(i_1)$\ $\|A_{n}v-Av\|_{X'}\le F_n(\|v\|_X+\delta_n)\quad\mbox{\rm for all}\ v\in X$;

$(i_2)$\ $\displaystyle\lim\, F_n=0$;

$(i_3)$\ $\delta_n\leq \widetilde{\delta}\ \ \mbox{\rm for all }n\in \mathbb N$.
\end{assumption}
\begin{assumption}\label{JJJJ}
 For each $u\in X,$ the application $X \ni v\to J(u,v)\in \mathbb R$ is weakly lower semicontinuous. In addition,
 for each sequence $\{u_n\}$ such that $u_n\rightharpoonup \widetilde{u}$ in $X$, the following inequalities hold:
 \begin{eqnarray}
 &&\limsup\big[ J(\widetilde{u},u_n-\widetilde{u})-J_n(\widetilde{u},u_n-\widetilde{u})\big]\leq 0;\label{nr1}\\ [2mm]
 &&\limsup \big[J_n(u_n,v-u_n)-J(u_n,v-u_n)\big]\leq 0\quad\mbox{\rm  for all }\ v\in X.
 \end{eqnarray}
 \end{assumption}
\begin{assumption}\label{bA}
There exists $M_0>0$ such that $M_n-m_n\geq M_0$\  for all $n\in\mathbb{N}$.
\end{assumption}
\begin{assumption}\label{qn}
There exists $\widetilde{q}>0$ such that $q_n\geq \widetilde{q}$\  for all $n\in\mathbb{N}$.
\end{assumption}
\begin{assumption}\label{bA2}
There exists $\widetilde{c}$ such that $c_n\leq \widetilde{c}$\ \,for all $n\in\mathbb{N}$.
\end{assumption}
\begin{assumption}\label{cvb}
For all sequences\ $\{z_n\}\subset X$,\ $\{\mu_n\}\subset Y$ \  \mbox{such that}
$z_n\rightharpoonup z\ \mbox{in}\ X$,\ $\mu_n\rightharpoonup \mu\ \mbox{in}\ Y$,\ \mbox{we have}
$\limsup\,b_n(w-z_n,\mu_n)\leq b(w-z,\mu)$  for all  $w\in X.$
\end{assumption}
\begin{assumption}\label{bb}
There exists $\alpha_0>0$ such that $\alpha_n\ge \alpha_0$ for all $n\in\mathbb{N}.$
\end{assumption}
\begin{assumption}\label{cvLa}
$ \{\Lambda_n\}$\ \mbox{\rm converge to}\ $\Lambda$\ \mbox{\rm in the sense of Mosco, i.e.,}\\ [3mm]
$(i_1)$\  for each $\mu\in \Lambda$ there exists a sequence $\{\mu_n\}$ such that
$\mu_n\in \Lambda_n$ for each $n\in \mathbb{N}$  and	$\mu_n\to \mu$ \mbox{in} $Y$;
\\[2mm]
$(i_2)$  for each sequence $\{\mu_n\}$ such that
\ $\mu_n\in\Lambda_n$ for each $n\in\mathbb{N}$  and $\mu_n\rightharpoonup \mu$ in $Y$, we have\ $\mu\in\Lambda$.
\end{assumption}
\begin{assumption}\label{cvf}
\ \ $ f_{n}\rightharpoonup f\quad {\rm in}\quad Z'$.
\end{assumption}
\begin{assumption}\label{cvpi}
	The operator $\pi$ is completely continuous, i.e., for each sequence\ $\{v_n\}\subset X$\  \mbox{such that}
	$v_n\rightharpoonup v\ \mbox{in}\ X$,\  \mbox{we have}
	\quad $$\pi v_n\to\pi v\quad{\rm in}\quad Y.
	$$
\end{assumption}

Our first result in this section states the convergence of the solution to Problem \ref{ppn} to the solution of Problem \ref{ap} and it is stated as follows.

\begin{theorem}\label{t2} Under Assumptions $\ref{a4}$--$\ref{a10}$, $\ref{a4}_n$--$\ref{a3bbiss}_n$, $\ref{a7}_n$--$\ref{a10}_n$, $\ref{cvA}$--$\ref{cvpi}$
the following convergences hold:
\begin{eqnarray}
&&\label{cvu}u_n\to u\qquad{\rm in}\quad X,\\
&&\label{cvla}\lambda_n\rightharpoonup \lambda\qquad{\rm in}\quad Y.
\end{eqnarray}
\end{theorem}
\begin{proof} The proof is carried out in several steps that we describe below.
\begin{step}\label{l1} We prove that the sequence $\{u_n\}$ is bounded in $X.$
\end{step}
Let $n\in\mathbb{N}$.
Using arguments similar to those used to obtain (\ref{nrmu})  we can write
\begin{equation}\label{mu}
M_n\|u_n\|_X^{q_n-1}\leq (c_0\|f_n\|_{Z'}+\|A_n0_X\|_{X'}+c_n).
\end{equation}
On the other hand, Assumptions \ref{cvA}, \ref{cvf} imply that
\begin{equation}\label{11}
\|A_n0_X\|_{X'}\le F_n\delta_n+\|A0_X\|_{X'},
\end{equation}
and there exist $\widetilde{F}>0$ and $\widetilde{f}>0$
such that
\begin{equation}\label{mfd}
|F_n|\leq \widetilde{F}\quad \mbox{and}\quad \|f_n\|_{Z'}\leq \widetilde{f}\quad \mbox{for all }n\in \mathbb N.
\end{equation}
Therefore, combining the inequalities (\ref{mu})--(\ref{mfd}) and keeping in mind Assumption \ref{bA2} we find that
\begin{equation*}
M_0\| u_n\|_X^{q_n-1}\leq  c_0\widetilde{f}+\widetilde{F}\widetilde{\delta}+\|A0_X\|_X+\widetilde{c}.
\end{equation*}
Next, using Assumption \ref{qn} we deduce that
\begin{equation}\label{bu}
\| u_n\|_X\leq k,
\end{equation}
where
\begin{equation}\label{defk}
k= M_0^{1/(1-\widetilde{q})}\big(c_0\widetilde{f}+\widetilde{F}\widetilde{\delta}+\|A0_X\|_X+\widetilde{c}\big)^{1/(\widetilde{q}-1)},
\end{equation}
which ends the proof of this step.


\begin{step}
We prove that the sequence $\{\lambda_n\}$ is bounded in $Y$.
\end{step}

First, we remark that $(\ref{bu})$ implies that $\{u_n\}\subset K$ where
\begin{equation}\label{defK}
K=\{\,v\in X\,|\, \|v\|\leq k\,\}.
\end{equation}
Let $n\in\mathbb{N}$. By using the inf-sup property of the form $b$ and arguments similar to those used in the proof of (\ref{nr}), we can write
\begin{equation}\label{13}
\alpha_n\|\lambda_n\|_Y\le  c_0\| f_n\|_{Z'}+\|A_n u_n\|_{X'}+c_n.
\end{equation}
On the other hand,
\begin{eqnarray*}
&&\|A_n u_n\|_{X'}\le \|A_nu_n-Au_n\|_{X'}+\|
Au_n\|_X\\ [2mm]
&&\quad\leq
F_n(\|u_n\|_X+\delta_n)+\|Au_n-A0_X\|_{X'}+\|A0_X\|_{X'}\\ [2mm]
&&\qquad\leq
\widetilde{F}(k+\widetilde{\delta})+L_{K} k+\|A0_X\|_{X'},
\end{eqnarray*}
where $K$ and $k$ are given by (\ref{defK}) and (\ref{defk}), respectively.

Therefore,
\begin{equation}\label{mla}
\alpha_0\|\lambda_n\|_Y\leq \,(c_0\widetilde{f}+\widetilde{F}(k+\widetilde{\delta})+L_{K}\,k+\|A0_X\|_{X'}+\widetilde{c}),
\end{equation}
which ends the proof of this step.

\begin{step}\label{sau}We prove that there exists  a pair  $(\widetilde{u},\widetilde{\lambda})\in X\times Y$ such that, passing to a subsequence still denoted $\{(u_n,\lambda_n)\}$,
we have $u_n \to {\widetilde{u}}$ in $X$ and
$\lambda_n \rightharpoonup {\widetilde{\lambda}}$  in $Y$.
\end{step}
The existence of an element $(\widetilde{u},\widetilde{\lambda})\in X\times Y$
as well as the weak convergences
$u_n \rightharpoonup {\widetilde{u}}$ in $X$ and
$\lambda_n \rightharpoonup {\widetilde{\lambda}}$  in $Y$ follows from Steps 1 and 2 combined with a standard
reflexivity argument. In order to prove the strong convergence  $u_n \to {\widetilde{u}}$ in $X$
we start by testing in (\ref{1n}) with $v=\widetilde{u}$. We have
\[(A_n u_n,\widetilde{u}-u_n)_{X',X}+b_n(\widetilde{u}-u_n,\lambda_n)+J_n(u_n,\widetilde{u}-u_n)\geq (f_n,\pi(\widetilde{u}-u_n))_{Z',Z}\]
and, therefore,
\begin{eqnarray*}
&&(A_n u_n-A_n\widetilde{u},u_n-\widetilde{u})_{X',X}-J_n(u_n, \widetilde{u}-u_n)-J_n(\widetilde{u},u_n-\widetilde{u})\leq\\
&&\quad(A_n\widetilde{u},\widetilde{u}-u_n)_{X',X}+b_n(\widetilde{u}-u_n,\lambda_n)+(f_n,\pi u_n-\pi\widetilde{u})_{Z',Z}\nonumber
-J_n(\widetilde{u},u_n-\widetilde{u}).\nonumber
\end{eqnarray*}

Moreover, using Assumptions $\ref{a4}_n$, $\ref{a7}_n$ $\ref{a7bis}_n$ and the inequalities  $M_n-m_n\geq M_0$ and $q_n\geq \widetilde{q},$ guaranteed by Assumptions \ref{bA} and \ref{qn}, respectively, we find that
\begin{eqnarray}\label{eutill}
M_0\|u_n-\widetilde{u}\|_X^{\widetilde{q}}\leq(A_n\widetilde{u},\widetilde{u}-u_n)_{X',X}+b_n(\widetilde{u}-u_n,\lambda_n)+(f_n,\pi u_n-\pi\widetilde{u})_{Z',Z}
-J_n(\widetilde{u},u_n-\widetilde{u}).
\end{eqnarray}

Next, we use Assumption \ref{cvA} ($i_1$) and write
\begin{eqnarray*}
	&&(A_n\widetilde{u},\widetilde{u}-u_n)_{X',X}=
	(A_n\widetilde{u}-A\widetilde{u},\widetilde{u}-u_n)_{X',X}+
	(A\widetilde{u},\widetilde{u}-u_n)_{X',X}\\ [2mm]
	&&\qquad\qquad \leq F_n(\|\widetilde{u}\|_X+\delta_n)\|u_n-\widetilde{u}\|_X+
	(A\widetilde{u},\widetilde{u}-u_n)_{X',X},
\end{eqnarray*}
then we pass to the upper limit in this inequality, and use Assumption \ref{cvA} $(i_2)$, $(i_3)$ to deduce that
\begin{equation}\label{17}
\limsup\, (A_n\widetilde{u},\widetilde{u}-u_n)_{X',X}\leq 0.
\end{equation}
Notice also that Assumptions \ref{cvf}, \ref{cvb} together with the convergences $u_n \rightharpoonup {\widetilde{u}}$ in $X$ and
$\lambda_n \rightharpoonup {\widetilde{\lambda}}$  in $Y$
imply that
\begin{eqnarray}
&&\label{16}
(f_n,\pi u_n-\pi\widetilde{u})_{Z',Z}\to 0,\\ [2mm]
&&\label{18}
\limsup\, b_n(\widetilde{u}-u_n,\lambda_n)\le 0.
\end{eqnarray}
Next, we use Assumption \ref{JJJJ} to write
\begin{eqnarray*}
	&& \limsup\, \big[-J_n(\widetilde{u},u_n-\widetilde{u})
	\big]\\ [2mm]
	&&=\limsup\,\big[J(\widetilde{u},u_n-\widetilde{u})-J_n(\widetilde{u},u_n-\widetilde{u})-J(\widetilde{u},u_n-\widetilde{u})\big]\\  [2mm]
	&&\leq\limsup\,\big[J(\widetilde{u},u_n-\widetilde{u})-J_n(\widetilde{u},u_n-\widetilde{u})]+\limsup[-J(\widetilde{u},u_n-\widetilde{u})\big]
	\\  [2mm]
	&&\leq -\liminf J(\widetilde{u},u_n-\widetilde{u})\leq -J(\widetilde{u},0_X).
\end{eqnarray*}
Using now Assumption \ref{J} $(i_2)$ we have $J(\widetilde{u},0_X)=0$ and, therefore,
\begin{equation} \limsup\, \Big[-J_n(\widetilde{u},u_n-\widetilde{u})\Big]\leq 0.
\end{equation}
We now pass to the upper limit in the inequality (\ref{eutill}) and we obtain
\[\limsup\,M_0\|u_n-\widetilde{u}\|^{\widetilde{q}}_X\leq 0.\]
This implies that $u_n\to\widetilde{u}$ in $X$ and concludes the proof of this step.

\begin{step} We prove that the pair  $(\widetilde{u},\widetilde{\lambda})\in X\times Y$ is a solution of Problem $\ref{ap}$.
\end{step}
First, we  recall that for each $n\in\mathbb{N}$ we have $\lambda_n\in\Lambda_n$. Keeping in mind Assumption \ref{cvLa}\,$(i_2)$ we deduce that
\begin{equation}
\label{19}\widetilde{\lambda}\in\Lambda.
\end{equation}
On the other hand, recall that $\{u_n\}\subset K$ where $K$
is the  closed  subset of $X$ defined by (\ref{defK}). Therefore, since $u_n\to\widetilde{u}$ in $X$, we deduce that  $\widetilde{u}\in K.$

Let $n\in\mathbb{N}$ and $v\in X.$ We use Assumption \ref{cvA} to see that
\begin{eqnarray*}
	&&\|A_nu_n-A\widetilde{u}\|_{X',X}\le \|A_n u_n-Au_n\|_{X'}+\|Au_n-A\widetilde{u}\|_{X'}\\ [2mm]
	&&\quad\le
	F_n(\|u_n\|_X+\delta_n)+ L_K\,\|u_n-\widetilde{u}\|_X
\end{eqnarray*}
and, therefore,
\begin{eqnarray}\label{w1}
A_nu_n\to A\widetilde{u}\quad{\rm in}\quad X'.
\end{eqnarray}

By using Assumption \ref{cvb} with
$z_n=0_X$, $\mu_n=\lambda_n$, $w=v$, and then with $z_n=v$, $\mu_n=\lambda_n$ and $w=0_X$ we obtain
\[\limsup\,b_n(v,\lambda_n)\le b(v,\widetilde{\lambda})
\quad{\rm and}\quad
b(v,\widetilde{\lambda})\le\liminf\,b_n(v,\lambda_n), \]
respectively. These  inequalities show that
\begin{equation*}
b_n(v,\lambda_n)\to b(v,\widetilde{\lambda}).
\end{equation*}
Similarly, taking $w=0_X,$ $z_n=u_n$ and $\mu_n=\lambda_n,$  by Assumption \ref{cvb} we obtain,
\[\limsup\,b_n(-u_n,\lambda_n)\leq b(-\widetilde{u},\widetilde{\lambda}).\]
Consequently,
\begin{equation}\label{21}
\limsup\,b_n(v-u_n,\lambda_n)\leq b(v-\widetilde{u},\widetilde{\lambda}).
\end{equation}

Moreover, note that Assumptions \ref{cvf} and the convergence $u_n\rightharpoonup\widetilde{u}$ in $X$ imply that
\begin{equation}
\label{22}(f_n,\pi (v-u_n))_{Z',Z}\to (f,\pi (v-\widetilde{u}))_{Z',Z}.
\end{equation}
Next we write
\begin{eqnarray*}
&&\limsup J_n(u_n,v-u_n)\leq \limsup \big[J_n(u_n,v-u_n)-J(u_n,v-u_n)\big]\\
&&\qquad+\limsup J(u_n,v-u_n),
\end{eqnarray*}
and, therefore,  Assumptions \ref{JJJJ} and \ref{J}
yield
\begin{equation}\label{JJUTY}\limsup J_n(u_n,v-u_n)\leq J(\widetilde{u},v-\widetilde{u})\quad\mbox{ for all }v\in X.
\end{equation}
Keeping in mind (\ref{w1})-(\ref{JJUTY}), we pass to the upper limit in the inequality (\ref{1n}) to see that
\begin{equation}\label{23}
(A\widetilde{u},v-\widetilde{u})_{X',X}+b(v-\widetilde{u},\widetilde{\lambda})+J(\widetilde{u},v-\widetilde{u})\geq (f,\pi (v-\widetilde{u}))_{Z',Z}.
\end{equation}

Consider now an arbitrary element $\mu\in\Lambda$. Using Assumption \ref{cvLa} we know that there exists a sequence
$\{\xi_n\}$ such that
$\xi_n\in \Lambda_n$ for each $n\in \mathbb{N}$  and $\xi_n\to \mu$ in $Y$. This allows to use the inequality (\ref{2n}) to see that
\[b_n(u_n,\xi_n-\lambda_n)\le 0,\]
which implies that
\begin{equation}\label{24}
\liminf\,b_n(u_n,\xi_n-\lambda_n)\le 0.
\end{equation}
On the other hand, by Assumption \ref{cvb} with $w=0_X,z_n=u_n$ and $\mu_n=\xi_n-\lambda_n$ we deduce that
\begin{equation*}
\limsup\,b_n(-u_n,\xi_n-\lambda_n)\le b(-\widetilde{u},\mu-\widetilde{\lambda})
\end{equation*}or, equivalently,
\begin{equation}\label{25}
b(\widetilde{u},\mu-\widetilde{\lambda})\le \liminf\,b_n(u_n,\xi_n-\lambda_n).
\end{equation}
We  combine the inequalities (\ref{24}) and (\ref{25}) to find that
\begin{equation}\label{26}
b(\widetilde{u},\mu-\widetilde{\lambda})\le 0.
\end{equation}

Finally, we gather  (\ref{19}), (\ref{23}) and (\ref{26}) to  conclude the proof of this step.
\hfill$\Box$


\begin{step} We now prove the convergences $(\ref{cvu})$ and $(\ref{cvla})$.
\end{step}

Recall that Theorem \ref{t1} states the existence of a unique solution to Problem \ref{ap}, denoted $(u,\lambda)$. Therefore, it follows from Step 4 that ${\widetilde{u}} = u$ and $\widetilde{\lambda}=\lambda$.
A careful examination of the  steps 1--4 reveals the fact that the sequence $\{(u_n,\lambda_n)\}$ is bounded in $X\times Y$
and every  subsequence of $\{(u_n,\lambda_n)\}$
which converges weakly in $X\times Y$ has the same limit $(u,\lambda)$. Therefore, by a standard argument we deduce  that the whole sequence $\{(u_n,\lambda_n\} \}$ converges weakly in $X\times Y$ to $(u,\lambda)$ or, equivalently,  $u_n \rightharpoonup{u}$ in $X$ and
$\lambda_n \rightharpoonup \lambda$  in $Y$.
This implies that  $(\ref{cvla})$ holds. Moreover, by repeating the arguments in Step 3 one shows that the strong convergence  $(\ref{cvu})$ holds, which
concludes the proof of the theorem.
\end{proof}

To proceed, we pay attention to the optimization of the solution to the mixed variational problem (\ref{pp:1})--(\ref{pp:2}). To this end, we consider a reflexive Banach space $W$  endowed with the norm $\|\cdot\|_W$ and a nonempty subset $U\subset W$.
For each $p\in U$ we consider an operator $A_p$, a form $b_p$,  a function $J_p$, a set $\Lambda_p$ and an element $f_p$  which satisfy Assumptions
$\ref{a4}$--$\ref{a3bbiss}$, $\ref{a7}$--$\ref{a10}$
with function $h_p$  and constants   $\alpha_p$, $c_p$, $M_p$, $m_p$, $q_p,$ $L_S^p$. To avoid any confusion, when used with $p$, we refer to these  assumptions as Assumptions $\ref{a4}_p$--$\ref{a3bbiss}_p$, $\ref{a7}_p$--$\ref{a10}_p$.
Then, if Assumption \ref{pi} is satisfied, we deduce from Theorem \ref{t1} that for each $p\in U$ there exists a unique solution $(u_p,\lambda_p)$ for the following problem.

\begin{problem}\label{ppp} Find $(u_p,\lambda_p)\in X\times\Lambda_p$
such that
	\begin{eqnarray}
	(A_pu_p,v-u_p)_{X',X}+b_p(v-u_p,\lambda_p)+J_p(u_p,v-u_p)&\geq& (f_p,\pi (v-u_p))_{Z',Z}\label{1p},\\
	b_p(u_p,\mu-\lambda_p)&\leq&0\label{2p}
	\end{eqnarray}
for all $v\in X,$ $\mu\in \Lambda_p$.	
\end{problem}
Consider also a cost functional
$J:U\to\mathbb{R}$ defined by
\begin{equation}\label{Jmn}
J(p)=	{\cal L}(u_p,\lambda_p,p)\qquad\mbox{\rm for all}\ \,p\in U
\end{equation}
where  ${\cal L}:X\times Y\times U \to\mathbb{R}$ is a given function which will be described below.
Then, the optimization problem we are interested in is the following.
\begin{problem}\label{o}  Find $p^*\in U$ such that
	\begin{equation}\label{o1p}
	J(p^*)=\min_{p\in U} J(p).
	\end{equation}
	\end{problem}

To solve Problem \ref{o} we consider the following assumptions.
\begin{assumption}\label{o2}
$ \qquad U\  \ \mbox{is a nonempty weakly closed subset of}\ \  W.$
\end{assumption}
\begin{assumption}\label{o3}
For all sequences $\{u_n\}\subset X,$ $\{\lambda_n\}\subset Y$  and
$\{p_n\}\subset U$ such that $u_n\rightarrow u$ in $X,$ $\lambda_{n}\rightharpoonup
\lambda$ in $Y,$ $p_{n}\rightharpoonup p$ in $W$, we have
$\displaystyle\liminf_{n\to
	\infty}\,{\cal L}(u_n,\lambda_n,p_n)\ge {\cal L}(u,\lambda,p).$
\end{assumption}
\begin{assumption}\label{o4} There exists $\psi: U\to\mathbb R$ such that:\medskip

$(i_1)$\ ${\cal L}(u,\lambda,p)\geq \psi(p)\quad \mbox{ for all }\,u\in X,\  \lambda\in Y,\ p\in U$;
\medskip

$(i_2)$\ $\|p_{n}\|_{W}\to+\infty\ \Longrightarrow\ \psi(p_n)\to \infty.$
\end{assumption}

\begin{assumption}\label{o5} $U$ is a bounded subset of $W$.
\end{assumption}

Our main result in this section is the following.

\begin{theorem}\label{t4}  Keep Assumptions
	$\ref{a4}_p$--$\ref{a3bbiss}_p$, $\ref{a7}_p$--$\ref{a10}_p$ for any $p\in U$, together with Assumptions $\ref{pi}$ , $\ref{cvpi}$, $\ref{o2},$ $\ref{o3}$  and either
	$\ref{o4}$ or $\ref{o5}$. In addition, assume that
	for any sequence $\{p_n\}\subset U$  such that $p_{n}\rightharpoonup p$ in $W,$
	Assumptions $\ref{cvA}$--$\ref{cvf}$ are satisfied
	with $A_n=A_{p_n}$, $A=A_p$,  $b_n=b_{p_n}$, $J_n=J_{p_n}$, $J=J_p$,   $\Lambda_n=\Lambda_{p_n}$, $\Lambda=\Lambda_p$, $f_n=f_{p_n}$, $f=f_p$.
	Then Problem $\ref{o}$ has  at least one solution.
\end{theorem}

\begin{proof}
Assume that $\{p_n\}\subset U$ is such that $p_n\rightharpoonup p\mbox{ in }W$. Since Assumptions  $\ref{cvA}$--$\ref{cvf}$
are satisfied in the sense prescribed in the statement of Theorem \ref{t4}, we are in a position to apply Theorem
\ref{t2} in order to obtain that
$u_{p_n}\to u_p$ in $X$ and  $\lambda_{p_n}\rightharpoonup \lambda_p$ in $Y$. Therefore, using the definition (\ref{Jmn})
and Assumption \ref{o3}  we deduce that
\begin{equation*}\label{r1}
\liminf J(p_n)=\liminf {\cal L}(u_{p_n},\lambda_{p_n},p_n)\ge {\cal L}(u_{p},\lambda_{p},p)=J(p).
\end{equation*}
It follows from here that the function $J:U\to\mathbb{R}$  is weakly lower semicontinuous.

If Assumption \ref{o4} is satisfied then,
for each sequence $\{p_n\}\subset U$, using ($i_1$)
we have
\[J(p_n)={\cal L}(u_{p_n},\lambda_{p_n},p_n)\ge \psi(p_n).\]
Therefore, if $\|p_n\|_{W} \to \infty$,  using ($i_2$) we deduce that
$J(p_n)\to \infty$ which shows that
$J:U\to\mathbb{R}$ is coercive.
Recall also Assumption \ref{o2} and the reflexivity of the space $W$. The existence of at least one solution to Problem \ref{o} is now a direct consequence of a minimization theorem of the Weierstrass-type. On the other hand, if Assumption \ref{o5} is satisfied, we are still in a position to apply  a Weierstrass argument, since now we minimize the function $J$ on a bounded set and, therefore, we do not need its coercivity.
 We deduce from here that, if either Assumption \ref{o4} or Assumption \ref{o5}
holds, then there exists at least one solution $p^*\in U$ to the optimization problem (\ref{o1p}), which concludes the proof.
\end{proof}

\section{An example}\label{apli}
The results  in the previous sections can be applied to the variational analysis of various nonlinear boundary value problems. To give an example, we consider here a nonlinear boundary value problem governed by the $r$-Laplace operator.
\begin{problem}\label{PM}
Find $u:{\bar{\Omega}}\to\mathbb{R}$ such that
\begin{eqnarray*} \label{p1:1}
\mu^*\,\operatorname{div}\,(\|\nabla u(\bx)\|^{r-2}\nabla u(\bx))&=&0\qquad\qquad\mbox{\rm in }\ \Omega, \\ \label{p1:2}
u(\bx)&=&0\qquad\quad\quad\mbox{\rm on}\ \Gamma_1,\\
 \mu^*\,\|\nabla u(\bx)\|^{r-2}\,\frac{\partial u}{\partial \nu}(\bx)&=&f(\bx)\,\,\,\quad\quad\mbox{\rm on }\ \Gamma_2,
 \end{eqnarray*}
 $$
\left.\begin{array}{l}
\left|\mu^*\,\|\nabla u(\bx)\|^{r-2}\,\displaystyle\frac{\partial u}{\partial \nu}(\bx)\right|\leq \vartheta,\\
 \mu^*\,\|\nabla u(\bx)\|^{r-2}\,\displaystyle\frac{\partial u}{\partial \nu}(\bx)=-\vartheta\,\frac{u(\bx)}{|u(\bx)|}\quad \mbox{\rm if }\ u(\bx)\neq 0
 \end{array}\right\}\,\,\quad \mbox{\rm on }\ \Gamma_3,
 $$
 $$
-\mu^*\,\|\nabla u(\bx)\|^{r-2}\,\frac{\partial u}{\partial \nu}(\bx)= g j(u(\bx))\,\,\quad \mbox{\rm on }\ \Gamma_4.
 $$
  \end{problem}
Here $\Omega$ is a bounded domain in $\mathbb R^2$ with smooth boundary $\Gamma$ partitioned in four measurable parts $\Gamma_i$, such that $meas(\Gamma_i)>0$,  $i\in\{1,2,3,4\}$.  Moreover, $\bnu$ denotes the unit outward normal on $\Gamma$, $r$ is a real number such that $2\leq r< \infty$, $\mu^*$, $\vartheta$, $g$ are given constants, and $f:\Gamma_2\to\mathbb{R}$,  $j:\mathbb R\to \mathbb R$ are given functions.

 Recall that Problem \ref{PM} represents a mathematical model which describes the frictional contact of an elastic cylinder with a rigid obstacle, in the antiplane framework, $\Omega$ being the cross section of the cylinder. Here $u$ represents the axial component of the displacement field, $f$ is related to the density of the surface tractions, $\vartheta$ and $g$ are given friction bounds and $j$ is a friction potential. Moreover, $r$ and $\mu^*$ are coefficients related to the constitutive law of the material, a law of
 Hencky-type, see, e.g., \cite{MMN18}. For the particular case $r=2$ and $\Gamma_4=\emptyset$ we refer the reader to \cite{SM09}. There, details on the mathematical treatment as well as  mechanical interpretations for the antiplane frictional contact models in a setting governed by variational inequalities of the second kind can be found.

In order to study Problem \ref{PM} we consider the space
\begin{equation*}\label{defX}
X=\{v\,:\, v\in W^{1,r}(\Omega),\,\gamma v(\bx)=0\mbox{ a.e. on }\Gamma_1\}
 \end{equation*}
where $\gamma:W^{1,r}(\Omega)\to L^r(\Gamma)$ is the trace operator. As it is known,  $\gamma$ is a linear continuous and compact operator. In particular, there exists $c_{tr}>0$ such that
\begin{equation}\label{ctr}
\|\gamma v\|_{L^{r}(\Gamma)}\leq c_{tr}\|v\|_{W^{1,r}(\Omega)}\qquad{\rm for\ all}\ \  v\in W^{1,r}(\Omega).
\end{equation}
Now, since
$r\geq 2,$ according to the trace theorem, $\gamma:W^{1,r}(\Omega)\to L^s(\Gamma)$ is also a linear continuous and compact operator, for all $s\ge 1$. These properties will be used repeatedly in this section, with $s=1$ or $s=r$, even if we do not mention it explicitly.

It is known that the space $X$ is a real reflexive Banach space endowed with the norm
\[\|u\|_X=\|\nabla u\|_{{L^r(\Omega)}^{2}}.\]
Moreover, we follow \cite{MMN18} and recall that there exists $c_P=c_P(\Omega,\Gamma_1)>0$
such that
\begin{equation}\label{cp}
\|u\|_{W^{1,r}(\Omega)}\leq c_P\|\nabla u\|_{L^r(\Omega)^2}.
\end{equation}

Let $r'$ be the conjugate exponent of $r,$ i.e. $\frac{1}{r}+\frac{1}{r'}=1$ and consider the real reflexive Banach spaces
\begin{equation}\label{defY}
\qquad Y=L^{r'}(\Gamma_3),\qquad Z=L^{r}(\Gamma_2).
\end{equation}
We denote by $\langle\cdot,\cdot\rangle$  the duality pairing between $L^{r'}(\Gamma_3)$ and $L^r(\Gamma_3)$, and by $\langle\cdot,\cdot\rangle_{Z',Z}$ the duality pairing between $Z$ and its dual $Z'=L^{r'}(\Gamma_2)$.


\medskip
The analysis of Problem \ref{PM} is made under the following assumptions.
\begin{assumption}\label{fd}
$f\in Z'.$
\end{assumption}
\begin{assumption}\label{fe}
$\mu^*>0,$ $\vartheta\ge 0,$ $g\ge 0.$
\end{assumption}
\begin{assumption}\label{j} The function $j:\mathbb R\to \mathbb R$ is nondecresing, bounded of rank $M_j>0$ and Lipschitz of rank $L_{j}>0.$
\end{assumption}
To give an example fulfilling Assumption \ref{j}, we can  consider the function \[j:\mathbb R\to \mathbb R \quad j(s)=\frac{s}{\sqrt{s^2+1}}\, .\]

Note that the nonhomogeneous case, in which $\mu^*$, $\vartheta$, $g$, and $j$ depend on the spatial variable $\bx$ can be considered. Nevertheless, for simplicity, we restrict below to the homogeneous case.

Next, we use ``$\cdot$" for the inner product in $\mathbb R^2$ and $\gamma v_{|_{\Gamma_2}}$, $\gamma v_{|_{\Gamma_3}}$
for the restriction of the trace of $v\in X$ to the parts $\Gamma_2$ and $\Gamma_3$ of the boundary of $\Omega$, respectively.
We also define the operator
$A:X\to X'$,
the form $b:X\times L^{r'}(\Gamma_3)\to \mathbb{R}$, the function  $J:X\times X\to\mathbb{R}$,
the operator $\pi:Z\to Z'$ and the set $\Lambda\subset Y$ by equalities
\begin{eqnarray}\label{defA}
&&(Au,v)_{X',X}=\mu^*\int_{\Omega}\|\nabla u(\bx)\|^{r-2}\nabla u(\bx)\cdot \nabla v(\bx)\,dx\quad \mbox{ for all }u,\,v\in X, \\ [3mm]
&&\label{defb}
 b(v,\zeta)=\langle\zeta,\gamma v_{|_{\Gamma_3}}\rangle\quad \mbox{ for all }v\in X, \ \zeta\in L^{r'}(\Gamma_3), \\ [3mm]
&&J(u,v)=g\int_{\Gamma_4} j(\gamma u(\bx)) \gamma v(\bx)\,d\Gamma\quad \mbox{ for all }u,\,v\in X, \label{defJ}\\ [3mm]
&&\pi v=\gamma v_{|_{\Gamma_2}}\quad \mbox{ for all }v\in X,\label{defpi}\\ [3mm]
&&\label{defLa}
\Lambda=\{\,\zeta\in Y\,:\,\langle\zeta, \gamma v_{|_{\Gamma_3}}\rangle\leq \vartheta\int_{\Gamma_3}|\gamma v(\bx)|\,d\Gamma\quad\mbox{ for all }v\in X\,\}.
\end{eqnarray}

 By standard arguments we can deliver the following mixed variational formulation of Problem \ref{PM}.
 \begin{problem}\label{sp}  Find $(u,\lambda)\in X\times \Lambda$ such that
\begin{eqnarray*} (Au,v-u)_{X',X}+b(v-u,\lambda)+J(u,v-u)
&\geq& (f,\pi (v-u))_{Z',Z}\quad{\rm for\ all}\
     v\in X,\\
b(u,\mu-\lambda)&\leq&0\qquad\  \qquad\quad\qquad{\rm for\ all}\ \mu\in \Lambda.
 \end{eqnarray*}
\end{problem}

 We have the following existence and uniqueness result.
 \begin{theorem}\label{t8} Under Assumptions $\ref{fd}$, $\ref{fe}$ and $\ref{j}$,  Problem $\ref{sp}$ has a unique solution.
 \end{theorem}
 \begin{proof}
We  apply Theorem  \ref{maini} (for the existence part) and Theorem \ref{t1} for (for the uniqueness part). To this end, we need to check the validity of the Assumptions \ref{a4}--\ref{unic}.

First, we use
 the basic properties of the $r$-Laplace operator (see \cite{DJM,GM75,MMN18}, for instance)  to see that
 Assumptions  \ref{a4},  \ref{a4bis},  \ref{coercive} and  \ref{a7} are fulfilled with $q=r$,
 $h(v)=\frac{\mu^*}{2^{r-2}r}\|v\|^r_X$
 and $M=\frac{\mu^*}{2^{r-2}r}.$
 Next, we follow \cite{M14} to recall that
 Assumptions \ref{a5} and \ref{a3bbiss} are satisfied. Moreover, using the properties of the trace operator  we see that
 Assumption  \ref{pi}  holds, too.

 Let us now verify  Assumption \ref{J}.
 First, condition $(i_1)$ is obviously satisfied. Next, we use the definition of $J$,  the properties of the trace operator and Assumption \ref{j} to see that
 \begin{eqnarray*}
 	&&|J(u,v)|\leq g M_j\int_{\Gamma_4}|\gamma v(\bx)|\,d\Gamma\le g M_j c_{tr} c_P\,{\rm meas}(\Gamma_4)^{\frac{1}{r'}}\|v\|_X,
 \end{eqnarray*}
for all $u,\, v\in X$. We conclude from here that condition $(i_2)$ in
Assumption \ref{J} is satisfied with
 $c=g M_j c_{tr} c_P\,{\rm meas}(\Gamma_4)^{\frac{1}{r'}}$. Recall that $M_j$, $c_{tr}$  and $c_P$ are the constants which appear in Assumption \ref{j}, (\ref{ctr}) and (\ref{cp}), respectively. Moreover, since
 \[J(u_n,v_n)-J(u,v)=g\left[\int_{\Gamma_4}\left(j(\gamma u_n)-j(\gamma u)\right)\gamma v_n\,d\Gamma-\int_{\Gamma_4} j(\gamma u)(\gamma v_n-\gamma v)\,d\Gamma\right],\]
  by using Assumption \ref{j}, as $\gamma$ is completely continuous (being linear and compact), we deduce that, if $u_n\rightharpoonup u$ and $v_n\rightharpoonup v$ in $X,$ then $J(u_n,v_n)\to J(u,v)$. Thus,
the point $(i_3)$  in Assumption \ref{J}  holds, too.

 Using  again Assumption \ref{j}, it follows that Assumption \ref{a7bis} is satisfied with $m=0.$ Obviously, Assumption \ref{sa} is fulfilled.
Finally, we observe that
\[J(u,v)+J(u,-v)=g\int_{\Gamma_4}j(\gamma u(\bx))\gamma v(\bx)\,d\Gamma-g\int_{\Gamma_4}j(\gamma u(\bx))\gamma v(\bx)\,d\Gamma=0\]
which shows that Assumption \ref{unic} is fulfilled, too. We also remark that Assumptions \ref{ef} and \ref{pi} are obviously satisfied.
Theorem \ref{t8} is now a direct consequence of Theorems \ref{maini} and \ref{t1}.
\end{proof}


 The solution of Problem \ref{sp} depends on the data $f$, $\vartheta$ and $g$ and, therefore, we denote it in what follows by $\big(u(f,\vartheta,g),\lambda(f,\vartheta,g)\big)$. Its dependence with respect to these data is provided by the following result.

 \begin{theorem} Let $\{f_{n}\}\subset Z'$, $\{\vartheta_n\}\subset [0,+\infty)$, $\{g_n\}\subset [0,\infty)$ be three sequences such that
 \begin{equation}\label{cv10}
 f_{n}\rightharpoonup f\quad{\rm in}\quad L^{r'}(\Gamma_2),\quad \vartheta_n\to \vartheta\quad{\rm in}\quad\mathbb R,\quad g_n\to g\quad{\rm in}\quad\mathbb R.
 \end{equation}
 Then,  under Assumptions $\ref{fd}$--$\ref{j}$, the following convergence hold:
 \begin{eqnarray}
 &&\label{cv11}
 u(f_{2n},\vartheta_n,g_n)\to u(f_2,\vartheta,g) \quad{\rm in}\quad X,\\ [2mm]
  &&\label{cv12}
  \lambda(f_{2n},\vartheta_n,g_n)\rightharpoonup \lambda(f_2,\vartheta,g) \quad{\rm in}\quad Y.
 \end{eqnarray}
 \end{theorem}

 \begin{proof}
We use Theorem \ref{t2} with $A_n=A$, $b_n=b$ and $J_n$, $\Lambda_n$ defined by
 	 \begin{eqnarray*}
 	&&J_n(u,v)=g_n\int_{\Gamma_4} j(\gamma u(\bx))\gamma v(\bx)\,d\Gamma\quad \mbox{ for all }u,\,v\in X, \\ [3mm]
 	&&\label{defL}
 	\Lambda_n=\{\,\zeta\in Y\,:\,\langle\zeta, \gamma v_{|_{\Gamma_3}}\rangle\leq \vartheta_n\int_{\Gamma_3}|\gamma v(\bx)|\,d\Gamma\quad\mbox{ for all }v\in X\,\}.
 \end{eqnarray*}	

Let us verify the validity of Assumptions \ref{a10}--\ref{cvpi}.

 First, we note that Assumption	\ref{a10} holds from the properties of the $r$-Laplacian operator, as shown in \cite{DJM,GM75}, for instance.
 Next, we note that, since $A_n=A$ and $b_n=b$, the  Assumptions \ref{cvA}, \ref{qn}, \ref{cvb}, \ref{bb} are obviously satisfied.
Moreover,  Assumption \ref{cvf} is a consequence of (\ref{cv10}) and Assumption \ref{cvpi} follows from the properties of the trace operator.

 Next  Assumption \ref{JJJJ} is a consequence of  Assumption \ref{j}, the properties of the trace operator and the convergence $g_n\to g$ in (\ref{cv10}).
 Indeed, by the boundedness of $j$ and the completely continuity of the trace operator,  it follows that for each $u\in X$, the function  $J(u,\cdot):X\to\mathbb{R}$  is weakly continuous. Assume now that  $u_n\rightharpoonup \widetilde{u}$ in $X$.  Then, using Assumption $\ref{j}$, for  each $n\in\mathbb{N}$ we have
  \begin{eqnarray*}
  &&|J(\widetilde{u},u_n-\widetilde{u})-J_n(\widetilde{u},u_n-\widetilde{u})|\\
  &&\quad\le|g-g_n|\int_{\Gamma_4}|j(\gamma \widetilde{u}(\bx)||\gamma u_n(\bx)-\gamma \widetilde{u}(\bx)|\,d\Gamma\\
   &&\qquad\leq M_j|g_n-g|\|\gamma \widetilde{u}-\gamma u_n\|_{L^1(\Gamma_4)}.
  \end{eqnarray*}
  Therefore,  the completely continuity of the trace operator  $\gamma: W^{1,r}(\Omega)\to L^1(\Gamma)$ combined with the convergences
  $u_n\rightharpoonup \widetilde{u}$ in $X$  and $g_n\to g,$ in $\mathbb{R}$ implies that
  \[J(\widetilde{u},u_n-\widetilde{u})-J_n(\widetilde{u},u_n-\widetilde{u})\to 0.\]
  By a similar argument,
  \begin{eqnarray*}
  &&|J_n(u_n,v-u_n)-J(u_n,v-u_n)|\\
  &&\qquad=|g_n-g|\int_{\Gamma_4}|j(\gamma u_n(\bx)| |\gamma v(\bx)-\gamma u_n(\bx)|\,d\Gamma\\
 &&\qquad\leq M_j|g_n-g|\|\gamma v-\gamma u_n\|_{L^1(\Gamma_4)}.
  \end{eqnarray*}
 Hence,  we are led to
  \[J_n(u_n,v-u_n)-J(u_n,v-u_n)\to 0.\]
  We conclude from here that  Assumption \ref{JJJJ} is verified.

  Moreover, Assumptions \ref{bA}  and \ref{bA2} are verified  since
  \[M_n=M=\frac{\mu^*}{2^{r-2}r},\quad m_n=0,
  \quad c_n=g_n M_j c_{tr} c_P\,{\rm meas}(\Gamma_4)^{\frac{1}{r'}},\] for each $n\in\mathbb{N}$. Note that Assumption \ref{qn} also holds, since $q=q_n=r$.

  On the other hand, Assumption  \ref{cvLa} is a consequence of the definitions of the sets $\Lambda_n$ and $\Lambda$, combined with the convergence $\vartheta_n\to \vartheta$ in (\ref{cv10}). Finally, Assumption \ref{cvf} is a consequence of the convergence $f_n\rightharpoonup f$ in (\ref{cv10}) and Assumption \ref{cvpi} is obviously satisfied.

It follows from above that we are in a position to apply Theorem \ref{t2} and, in this way, we deduce that the convergences (\ref{cv11}), (\ref{cv12}) hold, which concludes the proof.
 \end{proof}

 Besides the mathematical interest in the convergence results  (\ref{cv11}) and (\ref{cv12}), this results is important from mechanical point of view since it provides   the continuous dependence of the weak solution of Problem \ref{PM} with respect to the densities of the surface tractions, and the friction bounds.

We now provide three examples of optimization problems associated to  Problem \ref{sp} for which the existence result in Theorem \ref{t4} works.  Everywhere below we assume that $\mu^*>0$ and  $r\ge 2$ are given. In addition, the function $j$ is given as well, and satisfies Assumption \ref{j}.
The three problems we consider below
have a common feature and can be casted in the following general form.

\begin{problem}\label{opm}  Find  $p^*\in U$ such that
	\begin{equation}\label{o1m}
	{\cal L}(u_{p^*},\lambda_{p^*},p^*)=\min_{ p\in {U}} {\cal L}(u_p,\lambda_p,p).
	\end{equation}
\end{problem}

Here $U$ is a subset of a reflexive Banach space $W$,  ${\cal L}:X\times Y\times U \to\mathbb{R}$ is the cost functional and, for each $p\in U$,   $(u_p,\lambda_p)$ represents the solution of Problem \ref{sp} in which part of the data are related to the parameter $p$.
Both the set $U$, the space $W$, the functional ${\cal L}$  and the mapping $p\mapsto(u_p,\lambda_p)$
will change from example to example and, therefore,  will be described below.

\begin{example} Let $W=Y\times \mathbb R\times \mathbb R$, $U=Y\times[0,+\infty)\times[0,+\infty)$,
\begin{equation}
	{\mathcal L}(u,\lambda,p)={\alpha}\,\|u-u_d\|^r_X+{\beta}\,\|\lambda-\lambda_d\|^{r'}_Y+	{\delta}\,\|p\|^2_W,
\end{equation}
where $\alpha,\beta,\delta>0$ are given constants and
$(u_d,\lambda_d)\in X\times Y$ is a given optimal target. An element $p\in U$ is of the form $p=(f,\vartheta,g)$ and $(u_p,\lambda_p)$ denotes the solution of Problem $\ref{sp}$ with the data
$f,\,\vartheta,\, g$, $\mu^*$, $j$ and $r$. It is easy to see that in this case
Assumptions $\ref{o2},$  $\ref{o3}$  and
$\ref{o4}$ are satisfied.
Therefore, using Theorem $\ref{t4}$ we deduce the existence of at least one solution for the corresponding optimization problem $(\ref{o1m})$.
\end{example}

\begin{example} In this example, besides $\mu^*$, $j$ and $r$, we assume that  $\vartheta\ge0$ and $g\ge 0$ and are given. Let $W=U=Y$,
	\begin{equation}
	{\mathcal L}(u,\lambda,p)={\alpha}\,\|\gamma u_{|_{\Gamma_2}}-u_d\|_Z^2+
	{\delta}\,\|f\|^2_{Z'},
	\end{equation}
	where $\alpha,\delta>0$ are positive constants
	and $u_d\in Z$ is a given optimal target. An element $p\in U$ is of the form $p=f$ and $(u_p,\lambda_p)$ denotes the solution of Problem $\ref{sp}$ with the data
	$f$,\,$\vartheta,\, g,$ $\mu^*$, $j$ and $r$. It is easy to see that in this case
	Assumptions $\ref{o2},$  $\ref{o3}$  and
	$\ref{o4}$ are satisfied.
	Therefore, using Theorem $\ref{t4}$ we deduce the existence of at least one solution for the corresponding optimization problem $(\ref{o1m})$.
\end{example}

\begin{example} In this example, besides $\mu^*$, $j$ and $r$, we assume that $f\in Z'$, $\vartheta\ge 0$ and $\widetilde{g}>0$ are given. Let $W=\mathbb R$, $U=[0,\widetilde{g}]$,
	\begin{equation}\label{L}
	{\mathcal L}(u,\lambda,p)=\|\gamma u_{|_{\Gamma_3}}-u_d\|_Y
	\end{equation}
	where $u_d\in Y$ is a given target.
	An element $p\in U$ is of the form $p=g$ and $(u_p,\lambda_p)$ denotes the solution of Problem $\ref{sp}$ with the data $f$,\,$\vartheta,\, g$, $\mu^*$, $j$ and $r$. It is easy to see that in this case
	Assumptions $\ref{o2}$,  $\ref{o3}$  and
	$\ref{o5}$ are satisfied.
	Therefore, using Theorem $\ref{t4}$ we deduce the existence of at least one solution for the corresponding optimization problem $(\ref{o1m})$.
\end{example}

\vskip 10mm

\noindent\textbf{Acknowledgements} This project has received funding from the European Union's Horizon 2020
Research and Innovation Programme under the Marie Sklodowska-Curie
Grant Agreement No 823731 CONMECH.

\end{document}